\documentclass[reqno,11pt]{amsart}

\usepackage{a4wide}
\usepackage{amssymb, graphicx}
\usepackage{amsmath,amsthm,mathrsfs,amsfonts}
\usepackage{hyperref}
\usepackage{xcolor}
\usepackage{float}
\usepackage{pgf,tikz}
\usetikzlibrary{arrows}

\allowdisplaybreaks

\def \C {\mathbb{C}}

\def \R {\mathbb{R}}

\def \N {\mathbb{N}}
\def \zf {\mathfrak{F}} 
\DeclareMathOperator{\supp}{\text{supp}}
\DeclareMathOperator{\sgn}{sgn}

\def\d{\mathrm{d}}

\newtheorem{theorem}{Theorem}[section]
\newtheorem{lemma}[theorem]{Lemma}

\newtheorem{remark}[theorem]{Remark}

\newtheorem{corollary}[theorem]{Corollary}

\begin{document}
	
	\title[Stability of phase retrieval]{On the effect of zero-flipping on the stability of the phase retrieval problem in the Paley-Wiener Class}
	\author{Philippe Jaming, Karim Kellay \& Rolando Perez III}
	\address{Ph. Jaming, Univ. Bordeaux, CNRS, Bordeaux INP, IMB, UMR 5251,  F-33400, Talence, France}  
	\email{Philippe.Jaming@math.u-bordeaux.fr}
	\address{K. Kellay, Univ. Bordeaux, CNRS, Bordeaux INP, IMB, UMR 5251,  F-33400, Talence, France}  
	\email{kkellay@math.u-bordeaux.fr}
	\address{R. Perez III , Univ. Bordeaux, CNRS, Bordeaux INP, IMB, UMR 5251,  F-33400, Talence, France}
	\address{Institute of Mathematics, University of the Philippines Diliman, 1101 Quezon City, Philippines}
	\email{roperez@u-bordeaux.fr}
	
\begin{abstract}

In the classical phase retrieval problem in the Paley-Wiener class $PW_L$ for $L>0$, \textit{i.e.} to recover $f\in PW_L$ from $|f|$, Akutowicz, Walther, and Hofstetter independently showed that all such solutions can be obtained by flipping an arbitrary set of complex zeros across the real line. This operation is called zero-flipping and we denote by $\mathfrak{F}_a f$ the resulting function. The operator $\mathfrak{F}_a$ is defined even if $a$ is not a genuine zero of $f$, that is if we make an error on the location of the zero. Our main goal is to investigate the effect of $\mathfrak{F}_a$. We show that $\mathfrak{F}_af$ is no longer bandlimited but is still wide-banded. We then investigate the effect of $\mathfrak{F}_a$ on the stability of phase retrieval by estimating the quantity $\inf_{|c|=1}\|cf-\mathfrak{F}_af\|_2$. We show that this quantity is in general not well-suited to investigate stability, and so we introduce the quantity $\inf_{|c|=1}\|c\mathfrak{F}_bf-\mathfrak{F}_af\|_2$. We show that this quantity is dominated by the distance between $a$ and $b$.

\end{abstract}

	\maketitle 

\section{Introduction}

The phase retrieval problem refers to the recovery of the phase of a function $f$ from known data from the magnitude of $f$ and some constraints on $f$ usually expressed in terms of properties of some transforms of $f$. A typical example consists in the recovery of $f$ from $|f|$ and the knowledge that the Fourier transform of $f$ is compactly supported. These problems have been studied due to their physical applications such as in x-ray crystallography \cite{Mi1990}, optical imaging \cite{SECCMS2015}, microscopy \cite{DHF1975}, and astronomy \cite{DF1987}. However, until the turn of the century, little was known in the mathematics literature. Early work on this problem centered on describing the set of solutions and finding additional constraints that
can lead to significant reductions of the set of solutions, see {\it e.g.} Klibanov {\it et. al.} \cite{Kl1995} and the first author's papers \cite{Ja1999, Ja2014}.
In the last decade, this subject has seen a blooming interest thanks to the discovery of new algorithms based on convex optimization (see {\it e.g.} \cite{CESV2015,CLS2014,CSV2013,WAM2015}). This has in turn triggered interest in the issue of stability, which refers to continuous dependence of the solution from the given magnitude data. It has been shown that phase retrieval problems situated in finite dimensions are stable, however, this is not the case for infinite dimensions \cite{ADGY2019,CCD2016}. More precisely, the stability deteriorates whenever the dimension increases \cite{AG2017,CCD2016}. For more information on phase retrieval problems, we refer the reader to the survey articles \cite{CLS2014,Fi1982,GKR2020,LBL2002,Mi1990} which include detailed discussions on both the theoretical (\textit{e.g.} abstract formulations, additional constraints, stability) and the numerical aspects (algorithms), and some physical examples.

In order to simply explain how zero-flipping works, we recall a classical Fourier phase
retrieval problem which was solved independently by Akutowicz \cite{Aku1956,Aku1957}, Walther \cite{Wa1963}, and Hofstetter \cite{Ho1964}: given $f$ in the Paley-Wiener class, \textit{i.e.} $f\in L^2(\R)$ with compactly supported Fourier transform, the goal is to find all $g$ in the Paley-Wiener class such that
$$
	|g(x)|=|f(x)|,\qquad x\in\R.
$$
We summarize the proof of their solution. Recall that the Paley-Wiener theorem extends $f$ and $g$ to entire functions of finite order 1. Writing $|f(x)|^2=|g(x)|^2$ or
$$
	f(x)\overline{f(\bar x)}=g(x)\overline{g(\bar x)},\qquad x\in\R,
$$
we see that their extensions satisfy
\begin{equation}
	\label{eq:intro1}
	g(z)\overline{g(\bar z)}=f(z)\overline{f(\bar z)},\qquad z\in\C.
\end{equation}
Since $f$ is of finite order, we can use Hadamard factorization theorem which states that entire functions of finite order are identified by their zeros. Here, we may write $f$ as
$$
	f(z)=ce^{\alpha z}z^m\prod_{k\in\N}\left(1-\dfrac{z}{z_k}\right)e^{z/z_k},\qquad z\in\C
$$
where $c,\alpha\in \C$, $m\in \N\cup \{0\}$ and $\{z_k\}_{k\in\N}$ is the sequence of nonzero zeros of $f$. Hence if we denote by $\mathcal{Z}(g)$ the zero set of $g$ (counting with multiplicities), by \eqref{eq:intro1} we have
$$
	\mathcal{Z}(g)\setminus\{0,0,...,0\}=\{z_k:k\in J\}\cup \{\overline{z_k}:k\in \N\setminus J\} ,\qquad J\subseteq \N.
$$
This process was called \textit{zero-flipping} by Walther. With this, it follows that all such $g$'s have Hadamard factorization given by
$$
	g(z)=\tilde ce^{(\alpha+i\gamma) z}z^m\prod_{k\in J}\left(1-\dfrac{z}{z_k}\right)e^{z/z_k}\prod_{k\in \N\setminus J}\left(1-\dfrac{z}{\overline{z_k}}\right)e^{z/\overline{z_k}},\qquad z\in\C
$$
where $|\tilde c|=|c|$ and $\gamma\in\R$. The convergence of the infinite product above to an entire function of order 1 is guaranteed by a result from Titchmarsh \cite{Ti1926}. Moreover, we see that $g\in L^2(\R)$ since $|g|=|f|$. For a more technical discussion of zero-flipping in this context, we refer the reader to the book of Hurt \cite[Section 3.17]{Hu1989}. 

Now, let $a\in \C\setminus\R$ and $f$ be in the Paley-Wiener class. Define the \textit{flipping operator}, denoted by $\mathfrak{F}_a$ where
$$
	\mathfrak{F}_a: f \longmapsto \dfrac{(1-x/\bar a)}{(1-x/a)}\dfrac{e^{x/\bar a}}{e^{x/a}}\cdot f,\qquad x\in\R.
$$
This operator exhibits the zero-flipping of $f$ at $a$. Indeed, dividing $f$ by the factor $(1-x/a)e^{x/a}$ cancels the canonical factor associated to $a$ while multiplying the result by $(1-x/\bar a)e^{x/\bar a}$ completes the flipping process. Whenever $f(a)= 0$, $\mathfrak{F}_af$ is still inside the Paley-Wiener class and is always a solution of the phase retrieval problem. On the other hand, when $f(a)\ne 0$, $\mathfrak{F}_af$ no longer belongs to the Paley-Wiener class. However we will show that $\mathfrak{F}_af$ is \textit{wide-banded}, that is, its Fourier transform statisfies a square-integrability condition with an exponential weight. It turns out that this problem was solved in our previous work in \cite{JKP2020}. 

The main question we address in this paper is that of stability of zero-flipping.
In some previous work on the stability of phase retrieval problems (see e.g. \cite{ADGY2019, GR2019}), stability was shown by finding (in some cases) a positive constant $C$ such that
\begin{equation}
	\label{eq:stable}
	\inf_{|c|=1}||f-cg||_{\mathcal B}\leq C\big|\big| |f|-|g| \big|\big|_{\mathcal B'}
\end{equation}
where $\mathcal B, \mathcal B'$ are suitable Banach or Hilbert spaces. Some error terms
may eventually be added. In our case, stability of the phase retrieval problem in 
some subclass $X$ of the Paley-Wiener class would mean that
$$
	\inf_{|c|=1}||f-cg||_2\leq C\big|\big| |f|-|g| \big|\big|_2+\mbox{(error term)}
$$
for every $f\in X$ and every solution $g\in X$ of the phase retrieval problem. In particular,
for $g=\mathfrak F_af$ we should recover the error term only and stability
would imply that this error term be small. Our aim here is to investigate this issue, namely to
get an estimate of
$$
	\inf_{|c|=1}||f-c\mathfrak{F}_af||_{2}.
$$
It turns out that when $a$ is in some small region near the origin, then this quantity is actually large (close to $2||f||_2$) so that zero-flipping of such a zero leads to instabilities. On the other hand, we will show that the error term is small when we flip a zero that is large and close to the real axis so that such a flipping does not lead to instabilities.

In a second stage, we compare the effect of two zero-flipping, that is, we investigate
\begin{equation}
	\label{eq:stabmeasure}
	\inf_{|c|=1}||\mathfrak F_af - c\mathfrak F_bf||_2.
\end{equation}
For instance, if $a,b\in \C\setminus\R$ are such that $f(a)\ne 0$ and $f(b)=0$, then we are comparing a genuine solution of the phase retrieval problem in the Paley-Wiener class with a solution obtained after having made a mistake on the location of the zero. Note that $\mathfrak F_af$, $\mathfrak F_bf$ are solutions of the phase retrieval problem. Thus, if the phase retrieval problem were stable, then this quantity should be an error term since it is bounded by
$$
	\inf_{|c|=1}||\mathfrak F_af - cf||_2+\inf_{|c|=1}||\mathfrak F_bf - cf||_2
$$
and should thus be an error term. We will indeed obtain an upper bound of \eqref{eq:stabmeasure} of the form $C(f)\,\text{dist}(a,b)$ where $C(f)$ is a positive constant depending on $f$, and $\text{dist}(a,b)$ is some distance function depending on $a$ and $b$.


This paper is organized as follows. Section 2 provides a short summary about the Fourier transform and its properties relevant to the study. Section 3 is devoted to our stability results.

\section{Preliminaries}

For $f\in L^1(\R)$, we use the following normalized definition for the Fourier transform $\widehat{f}$ given by
$$
	\widehat{f}(w)=\dfrac{1}{\sqrt{2\pi}}\int_\R f(x)e^{-iwx}\,\d x,\qquad w\in\R.
$$
With this definition, we have Parseval's identity given by $||f||_2=||\widehat f||_2$ for $f\in L^2(\R)$. Recall also that for all $x,w\in\R$,
\begin{enumerate}
	\item if $g(x)=e^{iax}f(x)$ for some $\alpha\in\R$, then $\widehat g(w) = \tau_\alpha \widehat f(w)= \widehat f(w-\alpha)$
	\item if $g,h\in L^2(\R)$, then $\widehat{gh}=\widehat g*\widehat h$ where
	$$
	(\widehat g*\widehat h)(w)=\dfrac{1}{\sqrt{2\pi}}\int_\R \widehat g(s)\widehat h(w-s)\,\d s,\qquad w\in\R.
	$$
\end{enumerate}

Recall that whenever $f\in L^2(\R)$ with $\supp\widehat{f}\subseteq[-L,L]$ for some $L>0$, $f$ is said to be bandlimited and is contained in the Paley-Wiener class which we denote by $PW_L$. The space $PW_L$ is a closed linear subspace of $L^2(\R)$.

We also recall the Paley-Wiener theorem on the strip, that is, whenever $f\in L^2(\R)$ and $\lambda>0$, $\widehat f\in L^2(\R, e^{2\lambda|x|}\d x)$ where
$$
	L^2(\R, e^{2\lambda|x|}\d x)=\left\{F\text{ is measurable}: \int_\R |F(x)|^2e^{2\lambda|x|}\d x<+\infty\right\}
$$
if and only if $f$ belongs to the Hardy space on the strip $H^2_\tau(\mathcal S_\lambda)$ (see \cite{JKP2020} and references therein). Here, $H^2_\tau(\mathcal S_\lambda)$ is the collection of all holomorphic functions on the strip $\mathcal S_\lambda=\{z\in\C:|\operatorname{Im}z|<\lambda \}$ such that
$$
	||f||^2_{H^2_\tau(\mathcal S_\lambda)}=\sup_{|y|<\lambda}\int_\R |f(x+iy)|^2\,\d x<+\infty.
$$
Note also that if $f\in H^2_\tau(\mathcal S_\lambda)$, then $\widehat{f}$ has to be concentrated near the origin so that $f$ is wide-banded.

For any function $f$, we denote its reflection with respect to the 
$y$-axis by $Rf$ given by
$$
	Rf(x)=f(-x), \qquad x\in\R.
$$

Finally, recall that the $L^2$-modulus of continuity of $F\in L^2(\R)$, denoted by $\omega_2(F;h)$ for some $h>0$, is given by
$$
	\omega_2(F;h)=\sup_{|\eta|\leq h}\left(\int_\R |F(x-\eta)-F(x)|^2\,\d x\right)^{1/2}= \sup_{|\eta|\leq h}||\tau_{\eta} F-F||_2.
$$

Throughout the paper, we use the notation $C(\alpha_1,\ldots,\alpha_n)$ to denote 
a positive constant that depends only on $\alpha_1,\ldots,\alpha_n\in\C$. The constant may change from one line to the next.

\section{Results}

\subsection{The operator $\zf_a$}

Let $f$ belong to $PW_L$ and let $a\in\C$ such that $\operatorname{Im}a>0$. 
Define the \textit{flipping operator} which we denote by $\mathfrak{F}_a$ where
\begin{equation}
	\label{eq:ZFa}
	(\mathfrak{F}_af)(x)=\dfrac{1-x/\bar{a}}{1-x/a}\cdot\dfrac{e^{x/\bar{a}}}{e^{x/a}}\,f(x),\qquad x\in\R.
\end{equation}
It is easy to verify that $|\mathfrak{F}_af|=|f|$ on $\R$ and so $||\mathfrak{F}_af||_2=||f||_2$, and thus also $||\widehat{\mathfrak{F}_af}||_2=||\widehat{f}||_2$. Note that it will suffice to analyze the stability for $\mathfrak{F}_a$ when $\operatorname{Im}a>0$ since we can cover the case $\text{Im }a<0$ by looking at 
$\mathfrak F_{\bar a}$ since for $x\in\R$, $\mathfrak F_{\bar a}f(x)=\overline{\mathfrak F_a\bar f(x)}$.

Observe that $\mathfrak F_af$ extends into an meromorphic function and that if $f(a)\ne 0$, then $\mathfrak F_af$ has a pole at $a$ and so that $\mathfrak F_af\notin PW_L$. On the other hand, if $f(a)=0$, from the Hadamard factorization of $f$ we see that $\mathfrak F_af$ has the effect of replacing the zero at $z=a$ by a zero at $z=\bar a$, and that $\mathfrak F_af$ is still holomorphic. From the Paley-Wiener theorem, we conclude that $\mathfrak F_af\in PW_L$. However, when $f(a)\ne 0$, $\mathfrak F_a$ extends to a holomorphic function on a strip. More precisely:

\begin{lemma}
	Let $f\in PW_L$ and let $a\in\C$ such that $\operatorname{Im}a>0$ and $f(a)\ne 0$. Then the operator $\mathfrak F_a: PW_L\longrightarrow H^2_\tau(\mathcal S_\lambda)$ is bounded with
	$$
		||\mathfrak F_af||_{H_\tau^2(\mathcal S_\lambda)}<\left[1+\dfrac{2\operatorname{Im}a}{\operatorname{Im}a-\lambda}\right]e^{\frac{2(\operatorname{Im}a)^2}{|a|^2}}e^{L\lambda}||f||_2
	$$
	where $\lambda < \operatorname{Im}a$. In particular, $\widehat{\mathfrak F_a f}\in L^2(\R, e^{2\lambda|x|}\d x)$.
\end{lemma}

\begin{proof}
	For $x,y\in\R$ with $|y|<\lambda<\operatorname{Im}a$, observe that if $z=x+iy$,
	\begin{align*}
		\bigg|\dfrac{(x+iy)-\bar a}{(x+iy)- a}\bigg|&=\bigg|1-\dfrac{2i\operatorname{Im}a}{z-a}\bigg|\\
		&\leq 1+\dfrac{2\operatorname{Im}a}{\sqrt{(x-\operatorname{Re}a)^2+(y-\operatorname{Im}a)^2}}\\
		&< 1+\dfrac{2\operatorname{Im}a}{|y-\operatorname{Im}a|}\\
		&< 1+\dfrac{2\operatorname{Im}a}{\operatorname{Im}a-\lambda}
	\end{align*}
	and 
	$$
		\bigg|\dfrac{e^{(x+iy)/\bar{a}}}{e^{(x+iy)/a}}\bigg|=e^{-\frac{2\operatorname{Im}a}{|a|^2}y}<e^{\frac{2(\operatorname{Im}a)^2}{|a|^2}}.
	$$
	Moreover, since $\widehat{\tau_{-iy}f}(\xi)=\widehat f(\xi)e^{\xi y}$ for $y\in\R$ such that $|y|<\lambda<\operatorname{Im}a$ and for $\xi\in\R$, Parseval's identity implies that
	$$
		\int_\R |f(x+iy)|^2=\int_{-L}^{L} |\widehat f(\xi)|^2e^{2\xi y}\,\d \xi\leq e^{2L\lambda}||f||_2^2.
	$$
	Thus, if $y\in\R$ such that $|y|<\lambda<\operatorname{Im}a$, we have
	\begin{align*}
		\int_\R |(\mathfrak{F}_af)(x+iy)|^2\d x &= \int_\R \bigg| \dfrac{(x+iy)-\bar a}{(x+iy)- a}\cdot \dfrac{e^{(x+iy)/\bar{a}}}{e^{(x+iy)/a}}\bigg|^2 |f(x+iy)|^2\,\d x\\
		&< \left[1+\dfrac{2\operatorname{Im}a}{\operatorname{Im}a-\lambda}\right]^2e^{\frac{4(\operatorname{Im}a)^2}{|a|^2}}\int_\R |f(x+iy)|^2\,\d x\\
		&<\left[1+\dfrac{2\operatorname{Im}a}{\operatorname{Im}a-\lambda}\right]^2e^{\frac{4(\operatorname{Im}a)^2}{|a|^2}}e^{2L\lambda}||f||_2^2<+\infty.
	\end{align*}
	Taking the supremum for all $y$ such that $|y|<\lambda<\operatorname{Im}a$ yields the first result. The second result then follows from the Paley-Wiener theorem on the strip.
\end{proof}

We now compute the explicit form of the Fourier transform of $\mathfrak F_af$ which we will need for our results.

\begin{lemma}
	Let $f\in PW_L$ for some $L>0$ and let $a\in\C$ such that $\operatorname{Im}a>0$. For all $x\in\R$, 
	\begin{equation}
	\label{eq:ZFa-hat}
	(\widehat{\zf_af})(x)=\dfrac{a}{\bar{a}}\left[\widehat f(x-\beta_a)-(2\operatorname{Im}a)\int_0^{+\infty} e^{ias}\widehat f(x-\beta_a+s)\,\d s\right]
	\end{equation}
	where $\beta_a=\dfrac{2\operatorname{Im}a}{|a|^2}$.
\end{lemma}
\begin{proof}
	Consider the function $\gamma_a$ defined by
	\begin{equation}
		\label{eq:gamma}
		\gamma_a(x)=-\dfrac{\sqrt{2\pi}i}{2}\left[1+\sgn(x)\right]e^{iax},\qquad x\in\R.
	\end{equation}
	It is easy to check that $\gamma_a\in L^1(\R)$ with $||\gamma_a||_1=\dfrac{\sqrt{2\pi}}{\operatorname{Im}a}$. Then, for all $w\in\R$,
	\begin{align*}
		\widehat\gamma_a (w)&=\dfrac{1}{\sqrt{2\pi}}\int_\R -\dfrac{\sqrt{2\pi}i}{2}\left[1+\sgn(x)\right]e^{i(a-w)x}\,\d x\\
		&=-i\int_0^{+\infty} e^{i(a-w)x}\,\d x\\
		&=\dfrac{1}{a-w}.
	\end{align*}
	Now, for all $x\in \R$, write \eqref{eq:ZFa} as
	\begin{align*}
		(\mathfrak F_af)(x)&=\dfrac{1-x/\bar{a}}{1-x/a}\cdot e^{i\beta_ax}f(x)\\
		&=\dfrac{a}{\bar a}\left[1-\dfrac{a-\bar a}{a-x}\right]e^{i\beta_ax}f(x)\\
		&=\dfrac{a}{\bar a}\left[e^{i\beta_ax}f(x)-(2i\text{Im }a)e^{i\beta_ax}f(x)\widehat{\gamma}_a(x)\right].
	\end{align*}
	Then
	\begin{equation}
		\label{eq:ga-comp}
		(\widehat{\zf_af})(x)=\dfrac{a}{\bar{a}}\left[\tau_{\beta_a}\widehat{f}(x)-(2i\text{Im }a)\left(R\gamma_a*\tau_{\beta_a}\widehat f\right)(x)\right],\qquad x\in\R.
	\end{equation}
	Expanding this equation, we get
	\begin{align*}
		(\widehat{\zf_af})(x)&=\dfrac{a}{\bar{a}}\left[\widehat f(x-\beta_a)-\dfrac{2i\operatorname{Im}a}{\sqrt{2\pi}}\int_\R \gamma_a(-s)\widehat f(x-\beta_a-s)\,\d s\right]\\
		&=\dfrac{a}{\bar{a}}\left[\widehat f(x-\beta_a)-\dfrac{2i\operatorname{Im}a}{\sqrt{2\pi}}\int_\R \gamma_a(s)\widehat f(x-\beta_a+s)\,\d s\right]\\
		&=\dfrac{a}{\bar{a}}\left[\widehat f(x-\beta_a)-2\operatorname{Im}a\int_0^{+\infty} e^{ias}\widehat f(x-\beta_a+s)\,\d s\right]
	\end{align*}
	as claimed.
\end{proof}

\subsection{Stability between $\zf_af$ and $f$} 

In this section, we will give an estimate of
$$
\inf_{|c|=1}\|f-c\zf_af\|_2.
$$
This is a classical measure of stability for the phase retrieval problem. We are here investigating
how far zero-flipping drives us from the original function (up to the trivial solution $f\to cf$).
Recall that $\mathfrak F_af$ is bandlimited only if $f(a)=0$, this will however play no role here, that is we allow the solution $\mathfrak{F}_af$ to be wide-banded. This can also be considered as a simpler case of \eqref{eq:stabmeasure}, where $b$ is real so that $\zf_bf=f$. Our result here is the following:

\begin{theorem}
	\label{thm:stab1}
		Let $f\in PW_L$ for some $L>0$. Let $a\in\C$ such that $\operatorname{Im}a>0$ and $\beta_a=\displaystyle\frac{2\operatorname{Im}a}{|a|^2}$. Then 
	\begin{align}
	\label{eq:stab1-1}
		\bigg|\inf_{|c|=1}||\zf_af-cf||_2^2-2||f||_2^2\bigg|\leq 30	\bigr(L\operatorname{Im}a\bigr)||f||_2^2,&\qquad\text{if }\beta_a> 2L
	\intertext{and}
	\label{eq:stab1-2}
		\inf_{|c|=1}||\zf_af-cf||_2^2\leq 2\,\omega_2(\widehat f;\beta_a)||f||_2+8\sqrt{L\operatorname{Im}a}\,||f||_2^2,&\qquad\text{if }\beta_a\leq 2L.
	\end{align}
\end{theorem}

\begin{remark}
The actual bounds are a bit more precise, see \eqref{eq:c1a} and \eqref{eq:c2a} in the proof below. Note that $\beta_a=\dfrac{2\operatorname{Im}a}{|a|^2}=2L$ is equivalent to $(\operatorname{Re}a)^2+\left(\operatorname{Im}a-\tfrac{1}{2L}\right)^2=\tfrac{1}{4L^2}$ which represents a circle, with a hole at the origin, illustrated below (in blue).
\begin{center}
	\vspace{3mm}
	\begin{tikzpicture}
	\fill[fill=red!20!white] (-5,0) -- (5,0) -- (5,4) -- (-5,4) -- cycle;
	\fill[fill=gray!20!white] (0,1.5) circle (1.5);
	\draw[draw=blue,very thick] (0,1.5) circle (1.5);
	\draw (-5,0) -- (5,0);
	\draw (0,0) -- (0,4);
	\draw[blue,thick] (0,0) circle (2pt);
	\filldraw[black,thick] (0,1.5) circle (2pt);
	\filldraw[black,thick] (-4,0.3) circle (2pt);
	\filldraw[black,thick] (-2.4,2.4) circle (2pt);
	\filldraw[black,thick] (0.2,0.15) circle (2pt);
	\coordinate[label={0:{$a_1$}}] (theta) at (-4,0.3);
	\coordinate[label={0:{$a_2$}}] (theta) at (-2.4,2.4);
	\coordinate[label={45:{$a_3$}}] (theta) at (0.2,0.15);
	\coordinate[label={0:{$\frac{1}{2L}i$}}] (theta) at (0,1.5);
	\coordinate[label={[text=blue]0:{$\beta_a=2L$}}] (theta) at (1.5,1.5);
	\coordinate[label={90:{$\operatorname{Re}a$}}] (theta) at (4.6,0);
	\coordinate[label={0:{$\operatorname{Im}a$}}] (theta) at (0,3.7);
	\end{tikzpicture}
	\\
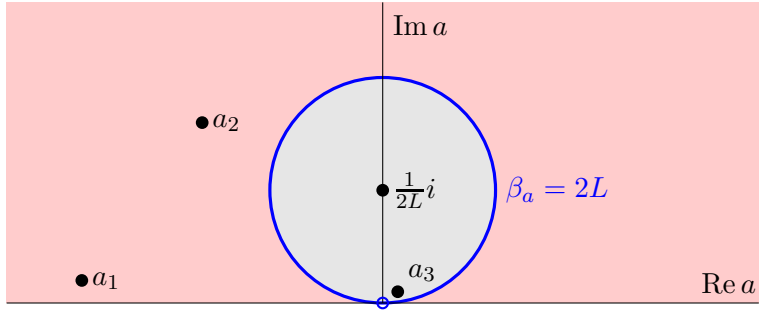
	\textnormal{The stability region}
	\vspace{3mm}
\end{center}
From Theorem \ref{thm:stab1}, we have stability when $\beta_a\leq 2L$ (in red), whereas we have instability when $\beta_a> 2L$ (in gray). Consider $a_1,a_2,a_3\in \C$ which have positive imaginary parts as plotted in the figure above. Note that the zero-flipping is `more stable' at $a_1$ as it is farther from the origin and has a smaller imaginary part than of $a_2$, and the zero-flipping at $a_3$ is unstable as it is very close to the origin.

This result says that zero-flipping becomes unstable (for this criteria)
when $a$ approaches the real axis inside this disc. On the other hand, if $a$
approaches the real line while staying away from the origin, we have stability.
Indeed, if $|a|\geq \alpha>0$ and $\operatorname{Im}a\longrightarrow 0$ then $\beta_a\longrightarrow 0$ so that
$\omega_2(\widehat f;\beta_a)\longrightarrow 0$.
\end{remark}

\begin{proof}[Proof of Theorem \ref{thm:stab1}]
	Observe first that for $|c|=1$, 
	\begin{align*}
		||\zf_af-cf||^2_2&=||\widehat{\zf_af}-c\widehat f||^2_2\\
		&=||\widehat{f}||_2^2-2\bar{c}\operatorname{Re}\langle\,\widehat{\zf_af},\widehat{f}\,\rangle+||\widehat{\zf_af}||_2^2\\
		&=2\left[||f||_2^2-\bar{c}\operatorname{Re}\langle\,\widehat{\zf_af},\widehat{f}\,\rangle\right],
	\end{align*}
	and thus
	\begin{equation}
	\label{eq:inf}
	\inf_{|c|=1}||
	\zf_af-cf||_2^2=2\left[||f||^2_2
	-\big\vert\langle\,\widehat{\zf_af},\widehat{f}\,\rangle\big\vert\right].
	\end{equation}
	
		For our calculations, we notice from \eqref{eq:ZFa-hat} that
		\begin{equation}
		\label{eq:ip}
		\dfrac{\bar a}{a}\langle\,\widehat{\zf_af},\widehat{f}\,\rangle=\sqrt{2\pi}\left(R\widehat{f}*\overline{\widehat{f}\,}\,\right)(\beta_a)-2\operatorname{Im}a\int_{-L}^L \int_0^{+\infty}e^{ias}\widehat f(x+s-\beta_a)\overline{\widehat f (x})\,\d s\,\d x\,
		\end{equation}
		with $R\widehat{f}(x)=\widehat f(-x)$ for $x\in\R$. 
		
	\medskip
	
		\noindent \textit{Case 1} $\beta_a>2L$. 
	
	\smallskip	
		
	Observe that if $\beta_a>2L$, then 
	\begin{equation}
	\label{eq:zff}
	\left(R\widehat{f}*\overline{\widehat{f}\,}\,\right)(\beta_a)=0.
	\end{equation}
	For the second term, if $\beta_a>2L$, then $x-\beta_a\leq L-\beta_a<-L$ for any $x\in[-L,L]$, thus 
	\begin{align*}
		\Bigg|\int_0^{+\infty}e^{ias}\widehat{f}(x+s-\beta_a)\,ds\Bigg|
		&=\Bigg|e^{-ia(x-\beta_a)}\int_{-L}^L e^{iat}\widehat{f}(t)\,\d t\Bigg|\\
		&\leq e^{\operatorname{Im}a\,(x-\beta_a+L)}||\widehat{f}||_1\\
		&\leq e^{\operatorname{Im}a\,(x-\beta_a+L)}\sqrt{2L}||f||_2
	\end{align*}
		and so 
	\begin{align}
		\big\vert\langle\,\widehat{\zf_af},\widehat{f}\,\rangle\big\vert&=
		\Bigg|2\operatorname{Im}a\int_{-L}^L\int_0^{+\infty}e^{ias}\widehat{f}(x+s-\beta_a)\overline{\widehat f(x)}\,\d s\,\d x\Bigg|\notag\\
		&\leq 2\operatorname{Im}a\left[\sqrt{2L}||f||_2\int_{-L}^L|\widehat{f}(x)|e^{\operatorname{Im}a(x-\beta_a+L)}\,\d x\right]\notag\\
		&\leq 2\sqrt{2L}\operatorname{Im}a||f||_2^2\left(\int_{-L}^L e^{2\operatorname{Im}a\,(x-\beta_a+L)}\,\d x\right)^{1/2}\notag\\
		&=2\sqrt{2L\operatorname{Im}a\cdot\sinh(2L\operatorname{Im}a)}e^{L\operatorname{Im}a }e^{-\beta_a\operatorname{Im}a}||f||_2^2.
		\label{eq:c1a}
	\end{align}
	Note that $\beta_a\operatorname{Im}a=\dfrac{2(\operatorname{Im}a)^2}{|a|^2}\leq 2$ so that $e^{-\beta_a\operatorname{Im}a}$
	plays no role and we just bound it by 1. Further, if $\beta_a>2L$ then $L\operatorname{Im}a <1$ thus $\sinh(L\operatorname{Im}a)\leq \sinh(1)L\operatorname{Im}a$. As $2\sqrt{2\sinh(2)}e^{1}\leq 15$ we get
	$$
		2\sqrt{2L\operatorname{Im}a\cdot\sinh(2L\operatorname{Im}a)}e^{L\operatorname{Im}a }\leq 15	L\operatorname{Im}a
	$$
	and finaly $\big\vert\langle\,\widehat{\zf_af},\widehat{f}\,\rangle\big\vert\leq 15	\bigr(L\operatorname{Im}a\bigr)||f||_2^2$.
	Together with \eqref{eq:zff}, we see that \eqref{eq:inf} imples \eqref{eq:stab1-1}.

	\medskip
	
		\noindent \textit{Case 2} $\beta_a\leq 2L$.	
	
	\smallskip 
	
	Observe first that
	\begin{align}
		||f||^2_2-\sqrt{2\pi}\left(R\widehat{f}*\overline{\widehat{f}\,}\,\right)(\beta_a)&\leq
		\Big|\sqrt{2\pi}\left(R\widehat{f}*\overline{\widehat{f}\,}\,\right)(\beta_a)-||f||_2^2\Big|\notag\\
		&=\Bigg|\int_{-L}^L\overline{\widehat{f}(\xi)}\left(\widehat{f}(\xi-\beta_a)-\widehat f(\xi)\right)\d\xi\Bigg|\notag\\
		&\leq ||f||_2||\widehat{f}-\tau_{\beta_a}\widehat{f}||_2\notag\\
		&\leq||f||_2\,\omega_2(\widehat{f};\beta_a).
		\label{eq:case2-1}
	\end{align}
	
	It remains to bound the second term in \eqref{eq:ip}, that is, to show that
	\begin{equation*}
		\left|2\operatorname{Im}a\int_{-L}^L \int_0^{+\infty}e^{ias}\widehat f(x+s-\beta_a)\,\d s\,\overline{\widehat f (x})\,\d x
		\right|\leq C(a)||f||_2^2
	\end{equation*}
	with $C(a)\longrightarrow 0$ when $\operatorname{Im}a\longrightarrow 0$.
	
	We first want to bound
		\begin{equation}
		\label{eq:2inner}
		\int_0^{+\infty}e^{ias}\widehat{f}(x+s-\beta_a)\,ds=e^{-ia(x-\beta_a)}\int_{x-\beta_a}^L e^{iat}\widehat{f}(t)\,\d t
		\end{equation}
	when $x\in[-L,L]$.
	To do so, assume first that $-L\leq x \leq 	\beta_a-L$. Then
	\begin{align*}
		\Bigg|\int_0^{+\infty}e^{ias}\widehat{f}(x+s-\beta_a)\,ds\Bigg|
		&\leq e^{\operatorname{Im}a\,(x-\beta_a)}\Bigg|\int_{-L}^L e^{iat}\widehat{f}(t)\,\d t\Bigg|\\
		&\leq e^{\operatorname{Im}a\,(x-\beta_a)}\Bigg|\int_{-L}^L e^{-\operatorname{Im}at}|\widehat{f}(t)|\,\d t\Bigg|\\
		&\leq e^{\operatorname{Im}a\,(x-\beta_a+L)}\sqrt{2L}||f||_2.
	\end{align*}
	On the other hand, if $\beta_a-L< x \leq L$, then by Cauchy-Schwarz inequality, we obtain
		\begin{align*}
		\Bigg|\int_0^{+\infty}e^{ias}\widehat{f}(x+s-\beta_a)\,ds\Bigg|&=\Bigg|e^{-ia(x-\beta_a)}\int_{x-\beta_a}^L e^{iat}\widehat{f}(t)\,\d t\Bigg|\\
		&\leq e^{\operatorname{Im}a\,(x-\beta_a)}||f||_2\left[\int_{x-\beta_a}^Le^{-2\operatorname{Im}a\,t}\,\d t\right]^{1/2}\\
		&=\dfrac{e^{\operatorname{Im}a\,(x-\beta_a)}||f||_2}{\sqrt{2\operatorname{Im}a}}\left[-e^{-2L\operatorname{Im}a}+e^{-2\operatorname{Im}a\,(x-\beta_a)}\right]^{1/2}\\
		&=\dfrac{||f||_2}{\sqrt{2\operatorname{Im}a}}\left[1-e^{2\operatorname{Im}a\,(x-\beta_a-L)}\right]^{1/2}.
		\end{align*}
	Hence, combining these two bounds, we have
	\begin{align}
		\Bigg|2\operatorname{Im}a\int_{-L}^L\int_0^{+\infty}e^{ias}\widehat{f}(x+s-&\beta_a)\overline{\widehat f(x)}\,\d s\,\d x\Bigg|\notag\\
		&\leq 2\operatorname{Im}a\Bigg[\sqrt{2L}||f||_2\int_{-L}^{\beta_a-L}|\widehat f(x)|e^{\operatorname{Im}a\,(x-\beta_a+L)}\,\d x\notag\\
		&\qquad+\dfrac{||f||_2}{\sqrt{2\operatorname{Im}a}}\int_{\beta_a-L}^L|\widehat{f}(x)|\left(1-e^{2\operatorname{Im}a\,(x-\beta_a-L)}\right)^{1/2}\d x\Bigg]\notag\\
		&\leq \sqrt{2\operatorname{Im}a}||f||_2^2\Bigg[2\sqrt{L\operatorname{Im}a}\left(\int_{-L}^{\beta_a-L}e^{2\operatorname{Im}a\,(x-\beta_a+L)}\,\d x\right)^{1/2}\notag\\
		&\qquad+\left(\int_{\beta_a-L}^L\left(1-e^{2\operatorname{Im}a\,(x-\beta_a-L)}\right)\d x\right)^{1/2}\Bigg]\notag\\
		&=\sqrt{2\operatorname{Im}a}||f||_2^2\Bigg[ \sqrt{2L}\left(1-e^{-2\beta_a\operatorname{Im}a}\right)^{1/2}\notag\\
		&\qquad+\left(2L-\beta_a-\dfrac{e^{-2\beta_a\operatorname{Im}a}}{2\operatorname{Im}a}(1-e^{2(\beta_a-2L)\operatorname{Im}a})\right)^{1/2}\Bigg].
		\label{eq:case2-2}
	\end{align}
	This gives the desired bound with
	\begin{equation}
		\small
		\label{eq:c2a}
		C(a)=\sqrt{2\operatorname{Im}a}\Bigg[ \sqrt{2L}\left(1-e^{-2\beta_a\operatorname{Im}a}\right)^{1/2}+\left(2L-\beta_a-\dfrac{e^{-2\beta_a\operatorname{Im}a}}{2\operatorname{Im}a}(1-e^{2(\beta_a-2L)\operatorname{Im}a})\right)^{1/2}\Bigg]
	\end{equation}
	and it is easy to see that $0<C(a)\leq4\sqrt{L\operatorname{Im}a}$. Finally, by \eqref{eq:case2-1} and \eqref{eq:case2-2}, we obtain the estimate in \eqref{eq:stab1-2}.
\end{proof}

\subsection{Stability between $\zf_af$ and $\zf_bf$}

In this section, we now compare two nontrivial solutions $\zf_af$ and $c\zf_bf$ of the phase retrieval problem. To do this, we introduce the following stability measure given by
$$
	\inf_{|c|=1}||\zf_af-c\zf_bf||_2.
$$
Note that we also allow these solutions to be either bandlimited or wide-banded. By using a similar computation we did to obtain \eqref{eq:inf}, note that
\begin{equation}
	\label{eq:inf2}
	\inf_{|c|=1}||\zf_af-c
	\zf_bf||_2^2=2\left[||f||^2_2-\big\vert\langle\,\widehat{\zf_af},\widehat{\zf_bf}\,\rangle\big\vert\right].
\end{equation}
Before we look at the next stability result, we first prove some technical lemmas which we will need.

\begin{lemma}
	\label{lem:techlem1}
	Let $f\in L^2(\R)$ and let $a,b\in \C$ with $\operatorname{Im}a,\operatorname{Im}b>0$ and $|a-b|\leq \dfrac{|b|}{2}$.
	Let $\beta_a=\displaystyle\frac{2\operatorname{Im}a}{|a|^2}$ and 
	$\beta_b=\displaystyle\frac{2\operatorname{Im}b}{|b|^2}$.
	 Consider $\gamma_a,\gamma_b\in L^1(\R)$ as defined in \eqref{eq:gamma}. Then
	$$
	||\operatorname{Im}a\,(R\gamma_a*\tau_{\beta_a}\widehat{f})-\operatorname{Im}b\,(R\gamma_b*\tau_{\beta_b}\widehat{f})||_2
	\leq C(a,b)||f||_2+\sqrt{2\pi}\,\omega_2(\widehat{f};\beta_a-\beta_b)
	$$
	where $C(a,b)\leq \displaystyle 14\frac{|a-b|}{\operatorname{Im}b}$.
\end{lemma}

\begin{remark} The actual value of $C(a,b)$ is a bit more precise and given in \eqref{eq:cab} below.
\end{remark}

\begin{proof}
	First, observe that for all $x\geq 0$,
	\begin{align*}
		\big|\sin(\tfrac{a-b}{2}x)\big|&\leq\big|\sin\left(\operatorname{Re}(\tfrac{a-b}{2}x)\right)\big|\cosh\left(\operatorname{Im}(\tfrac{a-b}{2})x\right)+\big|\sinh\left(\operatorname{Im}(\tfrac{a-b}{2})x\right)\big|\\
		&\leq\big|\operatorname{Re}(\tfrac{a-b}{2})\big|x\cdot e^{\frac{\operatorname{Im}(a-b)}{2}x}+\big|\sinh\left(\operatorname{Im}(\tfrac{a-b}{2})x\right)\big|.
	\end{align*}
	Hence, with this bound, we get
	\begin{align*}
		||R\gamma_a-R\gamma_b||_1&=\sqrt{2\pi}\int_0^{+\infty}|e^{iax}-e^{ibx}|\, \d x\\
		&=\sqrt{2\pi}\int_0^{+\infty}|e^{i\frac{a+b}{2}x}||e^{i\frac{a-b}{2}x}-e^{-i\frac{a-b}{2}x}|\,\d x\\
		&=2\sqrt{2\pi}\int_0^{+\infty}e^{-\frac{\operatorname{Im}(a+b)}{2}x}\big|\sin(\tfrac{a-b}{2}x)\big|\,\d x\\
		&\leq 2\sqrt{2\pi}\bigg[\int_0^{+\infty}\big|\operatorname{Re}(\tfrac{a-b}{2})\,\big|xe^{-\operatorname{Im}b\,x}\,\d x+\int_0^{+\infty}e^{-\frac{\operatorname{Im}(a+b)}{2}x}\big|\sinh\left(\operatorname{Im}(\tfrac{a-b}{2})x\right)\big|\,\d x\bigg]\\
		&=\sqrt{2\pi} \left[\dfrac{|\operatorname{Re}a-\operatorname{Re}b\,|}{(\operatorname{Im}b)^2}+\Big|\dfrac{1}{\operatorname{Im}a}-\dfrac{1}{\operatorname{Im}b}\Big|\right].
	\end{align*}
	Note also that $|a-b|\leq\dfrac{|b|}{2}$ implies that $\operatorname{Im}a\leq \dfrac{3}{2}\operatorname{Im}b$. Using this, the previous norm estimate, and Young's convolution inequality, we then have
	\begin{align*}
		||\operatorname{Im}a\,(R\gamma_a*\tau_{\beta_a}\widehat{f})&-\operatorname{Im}b\,(R\gamma_b*\tau_{\beta_b}\widehat{f})||_2\\
		&\leq |\operatorname{Im}a-\operatorname{Im}b\,|\cdot||R\gamma_b*\tau_{\beta_b}\widehat{f}||_2+\operatorname{Im}a\,||(R\gamma_a-R\gamma_b)*\tau_{\beta_b}\widehat{f}||_2\\
		&\qquad\qquad+\operatorname{Im}a\,|| R\gamma_a*(\tau_{\beta_a}\widehat{f}-\tau_{\beta_b}\widehat{f})||_2\\
		&\leq |\operatorname{Im}a-\operatorname{Im}b\,|\cdot ||R\gamma_b||_1||f||_2+\operatorname{Im}a\,||R\gamma_a-R\gamma_b||_1||f||_2\\
		&\qquad\qquad+\operatorname{Im}a\,||R\gamma_a||_1\,\omega_2(\widehat f;\beta_a-\beta_b)\\
		&=\sqrt{2\pi}\, \dfrac{|\operatorname{Im}a-\operatorname{Im}b\,|}{\operatorname{Im}b}||f||_2+\operatorname{Im}a\,||R\gamma_a-R\gamma_b||_1||f||_2+\sqrt{2\pi}\,\omega_2(\widehat f;\beta_a-\beta_b)\\
		&\leq\left[\sqrt{2\pi}\, \dfrac{|\operatorname{Im}a-\operatorname{Im}b\,|}{\operatorname{Im}b}+\dfrac{3\sqrt{2\pi}}{2}\operatorname{Im}b\left[\dfrac{|\operatorname{Re}a-\operatorname{Re}b\,|}{(\operatorname{Im}b)^2}+\Big|\dfrac{1}{\operatorname{Im}a}-\dfrac{1}{\operatorname{Im}b}\Big|\right]\right]||f||_2\\
		&\qquad\qquad+\sqrt{2\pi}\,\omega_2(\widehat f;\beta_a-\beta_b).
	\end{align*}
	We thus obtain the lemma with
	\begin{equation}
	\label{eq:cab}
		C(a,b)=\sqrt{2\pi}\, \dfrac{|\operatorname{Im}a-\operatorname{Im}b\,|}{\operatorname{Im}b}+\dfrac{3\sqrt{2\pi}}{2}\operatorname{Im}b\left[\dfrac{|\operatorname{Re}a-\operatorname{Re}b\,|}{(\operatorname{Im}b)^2}+\Big|\dfrac{1}{\operatorname{Im}a}-\dfrac{1}{\operatorname{Im}b}\Big|\right].
	\end{equation}
	Note that if $|a-b|\leq\dfrac{|b|}{2}$ then $\operatorname{Im}a\geq\dfrac{1}{2}\operatorname{Im}b$ thus
	$$
		\Big|\dfrac{1}{\operatorname{Im}a}-\dfrac{1}{\operatorname{Im}b}\Big|
		\leq 2\,\dfrac{|\operatorname{Im}a-\operatorname{Im}b\,|}{(\operatorname{Im}b)^2}
	$$
	from which the bound $C(a,b)\leq \displaystyle 14\frac{|a-b|}{\operatorname{Im}b}$ immediately follows.
\end{proof}

\begin{lemma} 
	\label{lem:techlem2}
	Let $f\in PW_L$ and let $b\in \C$ with $\operatorname{Im}b>0$ and $\beta_b=\dfrac{2\operatorname{Im}b}{|b|^2}$. Then
	\begin{equation*}
		\left[\int_\R\left(e^{\operatorname{Im}b\,(x-\beta_b)}\int_{x-\beta_b}^{L}e^{-\operatorname{Im}b\,y}|\widehat{f}(y)|\,\d y\right)^2\d x\right]^{1/2}\leq ||f||_2\,C(b)
	\end{equation*}
	with
	\begin{equation*}
		C(b)=\dfrac{1}{\sqrt{2\operatorname{Im}b}}\left[2L+1+\dfrac{e^{-4L\operatorname{Im}b}-1}{2\operatorname{Im}b}\right]^{1/2}.
	\end{equation*}
\end{lemma}

\begin{proof}
	Firstly, if $x\geq L+\beta_b$, then 
	$$
		\int_{x-\beta_b}^{L}e^{-\operatorname{Im}b\,y}|\widehat{f}(y)|\,\d y=0.
	$$
	Secondly, if $-L+\beta_b\leq x \leq L+\beta_b$, Cauchy-Schwarz inequality implies that
	\begin{align*}
		e^{\operatorname{Im}b\,(x-\beta_b)}\int_{x-\beta_b}^{L}e^{-\operatorname{Im}b\,y}|\widehat{f}(y)|\,\d y&\leq e^{\operatorname{Im}b\,(x-\beta_b)} ||f||_2\left(\int_{x-\beta_b}^{L} e^{-2\operatorname{Im}b\,y}\,\d y\right)^{1/2}\\
		&=e^{\operatorname{Im}b\,(x-\beta_b)}||f||_2\left(\dfrac{e^{-2\operatorname{Im}b\,(x-\beta_b)}-e^{-2\operatorname{Im}b\,L}}{2\operatorname{Im}b} \right)^{1/2}\\
		&=\dfrac{||f||_2}{\sqrt{2\operatorname{Im}b}}\left(1-e^{2\operatorname{Im}b\,x}e^{-2\operatorname{Im}b\,(L+\beta_b)}\right)^{1/2}
	\end{align*}
	and so 
	\begin{align*}
		\int_{-L+\beta_b}^{L+\beta_b} \bigg(e^{\operatorname{Im}b\,(x-\beta_b)}\int_{x-\beta_b}^{L}e^{-\operatorname{Im}b\,y}|\widehat{f}(y)|\,\d y\bigg)^2\d x&=\dfrac{||f||_2^2}{2\operatorname{Im}b}\int_{-L+\beta_b}^{L+\beta_b}\left(1-e^{2\operatorname{Im}b\,x}e^{-2\operatorname{Im}b\,(L+\beta_b)}\right)\d x\\
		&=\dfrac{||f||_2^2}{2\operatorname{Im}b}\left[2L+\dfrac{e^{-4L\operatorname{Im}b}-1}{2\operatorname{Im}b}\right].
	\end{align*}
	Lastly, if $x\leq -L+\beta_b$,
	\begin{align*}
		e^{\operatorname{Im}b\,(x-\beta_b)}\int_{x-\beta_b}^{L}e^{-\operatorname{Im}b\,y}|\widehat{f}(y)|\,\d y&=e^{\operatorname{Im}b\,(x-\beta_b)}\int_{-L}^{L}e^{-\operatorname{Im}b\,y}|\widehat{f}(y)|\,\d y\\
		&\leq e^{2\operatorname{Im}b\,(x-\beta_b)}e^{2\operatorname{Im}b\,L}||f||_2^2
	\end{align*}
	and thus,
	\begin{align*}
		\int_{-\infty}^{-L+\beta_b} \bigg(e^{\operatorname{Im}b\,(x-\beta_b)}\int_{x-\beta_b}^{L}e^{-\operatorname{Im}b\,y}|\widehat{f}(y)|\,\d y\bigg)^2\d x&\leq ||f||_2^2\,e^{2\operatorname{Im}b\,(L-\beta_b)} \int_{-\infty}^{-L+\beta_b} e^{2\operatorname{Im}b\,x}\,\d x\\
		&=\dfrac{||f||_2^2}{2\operatorname{Im}b}.
	\end{align*}
	Combining all these cases, we obtain
	\begin{equation*}
		\int_\R \left(e^{\operatorname{Im}b\,(x-\beta_b)}\int_{x-\beta_b}^{L}e^{-\operatorname{Im}b\,y}|\widehat{f}(y)|\,\d y\right)^2\d x\leq ||f||_2^2\,C(b)^2
	\end{equation*}
	where
	$$
		C(b)=\dfrac{1}{\sqrt{2\operatorname{Im}b}}\left[2L+1+\dfrac{e^{-4L\operatorname{Im}b}-1}{2\operatorname{Im}b}\right]^{1/2}
	$$
as announced.
\end{proof}
With these lemmas, we now state and prove our next stability result.
\begin{theorem}
	\label{thm:stab2}
	Let $f\in PW_L$ for some $L>0$. Let $a,b\in\C$ such that $\operatorname{Im}a,\operatorname{Im}b>0$, and $|a-b|\leq \dfrac{|b|}{2}$. Let $\beta_a=\displaystyle\frac{2\operatorname{Im}a}{|a|^2}$ and 
	$\beta_b=\displaystyle\frac{2\operatorname{Im}b}{|b|^2}$. Then
	$$
		\inf_{|c|=1}||\zf_af-c\zf_bf||_2^2\leq C_1(b)\,\omega_2(\widehat f;\beta_b-\beta_a)||f||_2+C_2(a,b)||f||_2^2.
	$$
	where $C_2(a,b)\longrightarrow 0$ as $a\longrightarrow b$.
\end{theorem}

\begin{remark}
The constants $C_1(b)$ and $C_2(a,b)$ are given in \eqref{eq:k1b}-\eqref{eq:k2b}
and depend on the quantities $C(a,b)$ and $C(b)$ given in Lemmas \ref{lem:techlem1} and \ref{lem:techlem2}, respectively. 	
\end{remark}

\begin{proof}
	We use the formula for $\zf_af$ from \eqref{eq:ga-comp}. Write
	$$
		\dfrac{\bar a b}{a\bar b}\langle\,\widehat{\zf_af},\widehat{\zf_bf}\,\rangle=\int_\R\left(\mathbf{A}_{a,b}(x)+\mathbf{B}_{a,b}(x)+\mathbf{C}_{a,b}(x)+\mathbf{D}_{a,b}(x)\right)\,\d x
	$$
	where 
	\begin{align*}
		\mathbf{A}_{a,b}(x)&=\tau_{\beta_b}\widehat{f}(x)\overline{\tau_{\beta_a}\widehat{f}(x)}\\
		\mathbf{B}_{a.b}(x)&=-\tau_{\beta_b}\widehat{f}(x)\overline{(2i\operatorname{Im}a)(R\gamma_a*\tau_{\beta_a}\widehat{f})(x)}\\
		\mathbf{C}_{a,b}(x)&=-\overline{\tau_{\beta_a}\widehat{f}(x)}{(2i\operatorname{Im}b)(R\gamma_b*\tau_{\beta_b}\widehat{f})(x)}\\
		\mathbf{D}_{a,b}(x)&=(4\operatorname{Im}a\operatorname{Im}b){(R\gamma_b*\tau_{\beta_b}\widehat{f})(x)}\overline{(R\gamma_a*\tau_{\beta_a}\widehat{f})(x)}
	\end{align*}
	for all $x\in\R$. Since $\int_\R \mathbf{A}_{b,b}(x)\,\d x=||f||_2^2$ and
	$$
		||f||_2^2=\langle\,\widehat{\zf_bf},\widehat{\zf_bf}\,\rangle=\int_\R\left(\mathbf{A}_{b,b}(x)+\mathbf{B}_{b,b}(x)+\mathbf{C}_{b,b}(x)+\mathbf{D}_{b,b}(x)\right)\,\d x,
	$$
	we have
	$$
		\int_\R\left(\mathbf{B}_{b,b}(x)+\mathbf{C}_{b,b}(x)+\mathbf{D}_{b,b}(x)\right)\,\d x=0.
	$$
	Hence,
	\begin{equation}
		\small
		\label{eq:thm2-4terms}
		\dfrac{\bar a b}{a\bar b}\langle\,\widehat{\zf_af},\widehat{\zf_bf}\,\rangle=\int_\R\Big[\mathbf{A}_{a,b}(x)+\left(\mathbf{B}_{a,b}-\mathbf{B}_{b,b}\right)(x)+\left(\mathbf{C}_{a,b}-\mathbf{C}_{b,b}\right)(x)+\left(\mathbf{D}_{a,b}-\mathbf{D}_{b,b}\right)(x)\Big]\,\d x.
	\end{equation}
	We will show the result by estimating each term of this integral.
	
	We first look at $\mathbf{A}_{a,b}$. Observe that
	\begin{align*}
		||f||_2^2-\int_\R \mathbf{A}_{a,b}(x)\,\d x&\leq \bigg|\int_\R \mathbf{A}_{a,b}(x)\,\d x-||f||_2^2\,\bigg|\\
		&\leq \int_\R |\widehat{f}(x-\beta_b)||\tau_{\beta_a}\widehat{f}(x)-\tau_{\beta_b}\widehat{f}(x)|\,\d x\\
		&\leq ||f||_2||\tau_{\beta_a}\widehat{f}-\tau_{\beta_b}\widehat{f}||_2\\
		&\leq ||f||_2\,\omega_2(\widehat f;\beta_a-\beta_b).
	\end{align*}
	For $\mathbf{B}_{a,b}-\mathbf{B}_{b,b}$, we use the bounds from Lemma \ref{lem:techlem1} to obtain
	\begin{align*}
		\int_\R |(\mathbf{B}_{a,b}-\mathbf{B}_{b,b})(x)|\,\d x&=2\int_\R |\tau_{\beta_b}\widehat f(x)||\operatorname{Im}a\,(R\gamma_a*\tau_{\beta_a}\widehat{f})(x)-\operatorname{Im}b\,(R\gamma_b*\tau_{\beta_b}\widehat{f})(x)|\,\d x\\
		&\leq 2||f||_2||\operatorname{Im}a\,(R\gamma_a*\tau_{\beta_a}\widehat{f})-\operatorname{Im}b\,(R\gamma_b*\tau_{\beta_b}\widehat{f})||_2\\
		&\leq 2C(a,b)||f||_2^2+2\sqrt{2\pi}\,\omega_2(\widehat{f};\beta_a-\beta_b)||f||_2.
	\end{align*}
	Next, we use the bounds from Lemma \ref{lem:techlem2} so that
	\begin{align*}
		\int_\R |(\mathbf{C}_{a,b}&-\mathbf{C}_{b,b})(x)|\,\d x\\
		&=2\operatorname{Im}a\int_\R |\tau_{\beta_a}\widehat{f}(x)-\tau_{\beta_b}\widehat{f}(x)||(R\gamma_b*\tau_{\beta_b}\widehat{f})(x)|\,\d x\\
		&=2\operatorname{Im}a\int_\R |\tau_{\beta_a}\widehat{f}(x)-\tau_{\beta_b}\widehat{f}(x)|\bigg|\int_0^{+\infty}e^{ibs}\widehat{f}(x-\beta_b+s)\,\d s\bigg|\,\d x\\
		&=2\operatorname{Im}a\int_\R |\tau_{\beta_a}\widehat{f}(x)-\tau_{\beta_b}\widehat{f}(x)|\bigg|e^{-ib(x-\beta_b)}\int_{x-\beta_b}^{L}e^{iby}\widehat{f}(y)\,\d y\bigg|\,\d x\\
		&\leq 2\operatorname{Im}a\int_\R |\tau_{\beta_a}\widehat{f}(x)-\tau_{\beta_b}\widehat{f}(x)|\left[e^{\operatorname{Im}b\,(x-\beta_b)}\int_{x-\beta_b}^{L}e^{-\operatorname{Im}by}|\widehat{f}(y)|\,\d y\right]\,\d x\\
		&\leq 2\operatorname{Im}b\,C(b)\cdot\omega_2(\widehat f;\beta_a-\beta_b)||f||_2.
	\end{align*}
	For the last term, the bounds from Lemmas \ref{lem:techlem1} and \ref{lem:techlem2} imply that
	\begin{align*}
		\int_\R |(\mathbf{D}_{a,b}&-\mathbf{D}_{b,b})(x)|\,\d x\\
		&=4\operatorname{Im}b\int_\R|(R\gamma_b*\tau_{\beta_b}\widehat{f})(x)||\operatorname{Im}a\,(R\gamma_a*\tau_{\beta_a}\widehat{f})(x)-\operatorname{Im}b\,(R\gamma_b*\tau_{\beta_b}\widehat{f})(x)|\,\d x\\
		&\leq 4\operatorname{Im}b \,C(b)||f||_2||\operatorname{Im}a\,(R\gamma_a*\tau_{\beta_a}\widehat{f})-\operatorname{Im}b\,(R\gamma_b*\tau_{\beta_b}\widehat{f})||_2\\
		&\leq 4\operatorname{Im}b\,C(b)C(a,b)||f||^2_2+4\sqrt{2\pi}\operatorname{Im}b\, C(b)\cdot\omega_2(\widehat f;\beta_a-\beta_b)||f||_2. 
	\end{align*}
	Combining these three estimates from above, we get
	\begin{align*}
		\bigg|\int_\R\Big[\left(\mathbf{B}_{a,b}-\mathbf{B}_{b,b}\right)&(x)+\left(\mathbf{C}_{a,b}-\mathbf{C}_{b,b}\right)(x)+\left(\mathbf{D}_{a,b}-\mathbf{D}_{b,b}\right)(x)\Big]\,\d x\bigg|\\
		&\leq\int_\R\Big[|\left(\mathbf{B}_{a,b}-\mathbf{B}_{b,b}\right)(x)|+|\left(\mathbf{C}_{a,b}-\mathbf{C}_{b,b}\right)(x)|+|\left(\mathbf{D}_{a,b}-\mathbf{D}_{b,b}\right)(x)|\Big]\,\d x\\
		&\leq\bigg[2\sqrt{2\pi}+(2+4\sqrt{2\pi})\operatorname{Im}b\,C(b)\bigg]\omega_2(\widehat f;\beta_a-\beta_b)||f||_2\\
		&\qquad\qquad+\bigg[2C(a,b)+4\operatorname{Im}b\,C(b)C(a,b)\bigg]||f||_2^2.
	\end{align*}
	Finally, from \eqref{eq:inf2}, we obtain
	\begin{align*}
		\inf_{|c|=1}||\zf_af-c\zf_bf||_2^2&\leq 2\Big|\dfrac{\bar a b}{a\bar b}\langle\,\widehat{\zf_af},\widehat{\zf_bf}\,\rangle-||f||_2^2\Big|\\
		&\leq \bigg[2+2\sqrt{2\pi}+(2+4\sqrt{2\pi})\operatorname{Im}b\,C(b)\bigg]\omega_2(\widehat f;\beta_a-\beta_b)||f||_2\\
		&\qquad\qquad+\bigg[2C(a,b)+4\operatorname{Im}b\,C(b)C(a,b)\bigg]||f||_2^2.
	\end{align*}
	Setting
	\begin{equation}
		\label{eq:k1b}
		C_1(b)=2+2\sqrt{2\pi}+(2+4\sqrt{2\pi})\operatorname{Im}b\,C(b)
	\end{equation}
	and
	\begin{equation}
		\label{eq:k2b}
		C_2(a,b)=2C(a,b)+4\operatorname{Im}b\,C(b)C(a,b),
	\end{equation}
	so that $C_2(a,b)\longrightarrow 0$ as $a\longrightarrow b$, we obtain the theorem.
\end{proof}

In this corollary, we see that if a false complex zero $a$ goes close to a genuine complex zero, then the corresponding wide-banded solution $\zf_af$ goes close to a genuine solution in the Paley-Wiener class.
\begin{corollary}
	Let $f\in PW_L$ for some $L>0$. Fix $b\in\C$, a simple zero of $f$ with $\operatorname{Im}b>0$. Suppose $a\in \C$ with $\operatorname{Im}a> 0$, $f(a)\ne 0$ and $|a-b|\leq \dfrac{|b|}{2}$. Then
	$$
		\inf_{|c|=1}||\zf_af-c\zf_bf||_2^2\longrightarrow 0\text{ as }a\longrightarrow b.
	$$
\end{corollary}
\begin{remark}
	Since zeros are isolated, $f$ does not vanish on $V\setminus\{b\}$ where $V$ is a neighborhood of $b$. Without loss of generality, $V\subset \left\{a\in\C: |a-b|\leq \dfrac{|b|}{2}\right\}$.
\end{remark}


\section*{Acknowledgements}
The research of the second author is partially supported by the project ANR-18-CE40-0035. The third author is supported by the CHED-PhilFrance scholarship from Campus France and the Commission of Higher Education (CHED), Philippines.

\section*{Data availability}

All data generated or analysed during this study are included in this published article

\IfFileExists{\jobname.bbl}{}
 {\typeout{}
  \typeout{******************************************}
  \typeout{** Please run "bibtex \jobname" to optain}
  \typeout{** the bibliography and then re-run LaTeX}
  \typeout{** twice to fix the references!}
  \typeout{******************************************}
  \typeout{}
 }

\end{document}